\documentclass[11pt,twoside]{article}
\usepackage{amsmath, amsthm, amssymb}
\usepackage{amsmath}
\usepackage{amsfonts}
\usepackage{dsfont}
\usepackage{epsfig}
\usepackage{algorithm}
\usepackage{algorithmic}
\usepackage{graphicx}
\usepackage{slashbox}
\usepackage{tikz}
\usepackage{amsmath, amsthm, amssymb}
\newtheorem{definition}{Definition}
\newtheorem{theorem}{Theorem}
\newtheorem{lemma}{Lemma}
\newtheorem{remark}{Remark}

\makeatletter \oddsidemargin0.45in \evensidemargin \oddsidemargin
\newcommand{\be}{\begin{equation}}
\newcommand{\ee}{\end{equation}}
\newcommand{\ba}{\begin{array}}
\newcommand{\ea}{\end{array}}
\newcommand{\bea}{\begin{eqnarray}}
\newcommand{\eea}{\end{eqnarray}}
\newcommand{\beas}{\begin{eqnarray*}}
\newcommand{\eeas}{\end{eqnarray*}}
\begin{document}

\title{Analysis of some  generalized $ABC-$ fractional logistic models}

\author{ Thabet Abdeljawad$^{a,d}$, Mohamed A. Hajji$^{b}$, \\Qasem M. Al-Mdallal$^{b}$ , Fahd Jarad $^{c}$,\\
$^{a}$Department of Mathematics and General Sciences,\\
Prince Sultan University\\ P. O. Box 66833,  11586 Riyadh, Saudi Arabia\\ Email:tabdeljawad@psu.edu.sa\\
 $^{b}$Department of Mathematical Sciences,
United Arab Emirates University\\
P.O. Box 15551, Al Ain, Abu Dhabi, UAE\\
Email: mahajji@uaeu.ac.ae; q.almdallal@uaeu.ac.ae\\
$^{c}$Department of Mahematics, \c{C}ankaya University, 06790 Ankara, Turkey\\
 Email: fahd@cankaya.edu.tr\\
$^d$  Department of Medical Research, China Medical University, 40402,\\ Taichung, Taiwan  }
\maketitle

\begin{abstract}
  In this article,  some logistic models in the settings of Caputo fractional operators with multi-parametered Mittag-Leffler kernels (ABC) are studied. This study mainly focuses on  modified quadratic and cubic logistic models in the presence of a Caputo type fractional derivative. Existence and uniqueness theorems are proved and stability analysis is discussed  by perturbing the equilibrium points. Numerical illustrative  examples are discussed for the studied models.\\
\\
\textbf{Keywords:} Generalized Mittag-Leffler kernel, generalized $ABC-$ fractional derivatives, generalized $ABC-$ fractional integrals , $ABC-$ logistic equations, modified $ABC-$ logistic model.
\end{abstract}

\section{Introduction}

One of the most recently popular branch  of mathematics is the fractional calculus that is concerned with derivatives and integrals of real or complex orders.  As a matter of fact, this type of calculus, although as old as the classic calculus, it has attracted the attention of researchers working on different disciplines because of the astounding results  obtained when some of these researchers exploited fractional operators in modeling some real world problems \cite{podlubny, Samko, Kilbas, Magin, Machado,Hilfer, Lorenzo,hajji2014efficient,al2010numerical}.  Despite this, these researchers have not stopped searching for more fractional operators, not only to enrich this calculus by discovering new kinds of fractional operators, but to understand better the complex systems they face in modeling. It can be realized that starting from the turn of this century researchers have proposed   a variety of fractional operators  \cite{had,fahd1,fahd2,Kat1,Kat2,fahd3,FTED 2017 ADIE} and added variety of fractional operators with different approaches to the field of fractional calculus. 

Till 2014, all the known fractional operators encapsuled singular kernels. In 2015, Caputo et al.  \cite{FCaputo} proposed a fractional derivative without singular kernel. This can be considered as a revolution in the theory of fractional calculus. This was the first step of the birth of  the most popular fractional derivative, Atangana-Baleanu fractional derivative \cite{Abdon}. After this derivative had been proposed, most of the works treated by the traditional fractional derivative were reconsidered by the new approach \cite{fahd, TD JNSA}. Afterwords,  Abdeljawad in \cite{T2,ThChaosAIP}, proposed a new nonsingular fractional derivative in Atanagana-Baleanu settings containing a multi-parametered Mittag-Leffler function and its discrete version in \cite{T1}. The entity of many parameters  redounded new properties to this derivative and enabled to overcome some obstacles.

Many researchers have discussed the logistic equation in the framework of differenet fractional derivatives (see \cite{El Sayed, Syed Abbas, Area-2016,TQF1}. In this work, we consider  two models of logistic equations in the frame of the recently proposed fractional derivative in \cite{T2,ThChaosAIP}.  We consider primarily  the  fractional  quadratic and cubic logistic models presented, respectively as 
\begin{equation}\label{mm1}
   ( ~_{t_0}^{ABC}D^{\theta,\mu,\gamma}x)(t)=r x(t) (1-x(t)),~~t>t_0,~x(t_0)=x_0,
  \end{equation}
  and
  \begin{equation}\label{mm2}
   ( ~_{t_0}^{ABC}D^{\theta,\mu,\gamma}x)(t)=r x(t) (1-\frac{x(t)}{k})(x(t)-m),~~t>t_0,~x(t_0)=x_0,
  \end{equation}
   $\theta \in (0,1]$, $\mu, \gamma >0$, and $r,m,k>0$. Here, $~_{t_0}^{ABC}D^{\theta,\mu,\gamma}$ represents the generalized  left $ABC-$  fractional derivative  introduced in \cite{T2} and  studied in \cite{ThChaosAIP}.

This article is organized as follows: In  section 2,  some essential concepts are presented.  The existence and uniqueness of the equations under consideration are discussed in section 3. In section 4, the stability of the considered logistic models  are discussed. In section 5,  numerical discussion is presented. The last section devoted to the conclusion.

\section{Preliminary results and essential concepts}
Motivated by the time scale notation we shall use the following modified versions of generalized Mittag-Leffler functions as was used previously, for example,  in \cite{T2,ThChaosAIP,T1}.

\begin{definition} \label{DML}  For $\lambda \in \mathbb{R},$ and $\theta, \beta,\rho, z \in \mathbb{C}$ with $Re(\theta)>0$, the generalized Mittag-Leffler functions are defined by
\begin{equation}
E_{\theta, \beta}^\rho(\lambda,z)= \sum_{k=0}^\infty \lambda^k
\frac{z^{k\theta+\beta-1}(\rho)_k} {\Gamma(\theta
k+\beta)k!}.
\end{equation}
For $\beta=\rho=1$, it is written that
\begin{equation} \label{nM22}
E_{\overline{\theta}} (\lambda, z)\triangleq E_{\theta, 1}^1(\lambda, z)=  \sum_{k=0}^\infty \lambda^k
\frac{z^{k\theta}} {\Gamma(\theta
k+1)},
\end{equation}
where $(\rho)_k=\rho (\rho+1)...(\rho+k-1)$. Notice that since $(1)_k=k!$,  then we write $ E^1_{\theta,\beta} ( \lambda,z)=E_{\theta,\beta} ( \lambda, z)$.
\end{definition}
\begin{definition}\cite{T2,ThChaosAIP} The  generalized left $ABC$ and $ABR$ fractional derivative starting at $t_0$ and with kernel $E^{\gamma}_{\theta, \mu}(\lambda,t)$, where $0<\theta<1,~~Re(\mu)>0,~\gamma \in \mathbb{R}$ and where $\lambda=  \frac{-\theta}{1-\theta}$,  are defined repspectively by
\begin{equation}\label{gABC}
(~^{ABC}_ {t_0}D^{\theta,\mu,\gamma} f)(t)=\frac{B(\theta)}{1-\theta}\int_{t_0}^t E^{\gamma}_{\theta, \mu}(\lambda,t-s)  f^\prime(s)ds,~~t \geq t_0.
\end{equation}
and
\begin{equation}\label{gABR}
(~^{ABR}_ {t_0}D^{\theta,\mu,\gamma} f)(t)=\frac{B(\theta)}{1-\theta}\frac{d}{dt}\int_{t_0}^t E^{\gamma}_{\theta, \mu}(\lambda,t-s)  f(s)ds,~~t \geq t_0.
\end{equation}

\end{definition}

\begin{definition} \cite{ThChaosAIP} \label{defn8}Assume $\eta(t)$ is defined on $[t_0,T]$. Then,
 the left generalized $AB$ fractional integral of order $0<\theta\leq1,\mu>0, \gamma>0$ is given by

 \begin{equation}\label{gammal}
 (~^{AB}_{a}I^{\theta,\mu,\gamma} \eta)(t)=\sum_{i=0}^\infty \binom{\gamma}{i} \frac{\theta^i}{B(\theta) (1-\theta)^{i-1}} (~_{a}I^{\theta i+1-\mu}\eta)(t).
\end{equation}
\end{definition}
For a continuous function $\eta(t)$  at $t_0$  whose ABR-derivative exists we know from \cite{T2} that

\begin{equation}\label{weknow}
  ( ~^{AB}_{a}I^{\theta,\mu,\gamma} ~^{ABR}_{a}D^{\theta,\mu,\gamma}\eta)(t)=\eta(t),~\texttt{and}~ (  ~^{ABR}_{a}D^{\theta,\mu,\gamma}~^{AB}_{a}I^{\theta,\mu,\gamma}\eta)(t)=\eta(t).
\end{equation}
The following lemma is (35) of  Theorem 3 in \cite{ThChaosAIP}. It shows the action of the generalized AB-integral operator on the generalized ABC-operator.
\begin{lemma} \label{lem1}
For $0< \theta< 1, ~~\mu >0,~~\gamma \in \mathbb{C}$, and $\lambda =\frac{-\theta}{1-\theta}$, we have
\begin{equation}\label{401}
   ( ~^{AB}_{a}I^{\theta,\mu,\gamma} ~^{ABC}_{a}D^{\theta,\mu,\gamma}\eta)(t)=\eta(t)-\eta(a).
\end{equation}

\end{lemma}
\begin{lemma}\label{lem2} \cite{T2,ThChaosAIP}
For any $0<\theta, ~\mu>0, ~~\gamma \in \mathbb{R}$, and $\chi$ defined for  $t\geq  t_0$, we have

 \begin{equation}
    (~^{ABC}_{t_0}D^{\theta,\mu,\gamma} \chi)(t)=(~^{ABR}_{t_0}D^{\theta,\mu,\gamma}\chi)(t)-\frac{B(\theta)}{1-\theta} \chi(a)E^\gamma_{\theta,\mu}(\lambda, t-a).
   \end{equation}
 Above $\lambda =\frac{-\theta}{1-\theta}$.
 \end{lemma}
The following lemma, which is Remark 8 in \cite{T1} is essential  to proceed, .
\begin{lemma} \label{tool1}
The continuous system
 \begin{equation}\label{clinear}
 (~^{ABC} _{t_0}D^{\theta,\mu,\gamma} \alpha)(t)=\rho \alpha(t)+g(t),~~~\alpha(t_0)=a_0,~~~0<\theta \leq 1,~\mu,\gamma \in \mathbb{C},~t \geq t_0,
 \end{equation}
 has the explicit solution

 \begin{eqnarray} \label{above}
\nonumber
  \alpha(t) &=& \alpha(t_0)\sum_{j=0}^\infty \rho^j (\frac{1-\theta}{B(\theta)})^j E_{\theta,j(1-\mu)+1}^{-\gamma j}(\lambda, t-t_0)  \\ \nonumber
  &+& g(t)*\sum_{j=0}^\infty \rho^j (\frac{1-\theta}{B(\theta)})^{j+1} E_{\theta,(j+1)(1-\mu)}^{-\gamma (j+1)}(\lambda, t-t_0)\\ \nonumber
   &=& \alpha(t_0)+ \alpha(t_0)\sum_{j=1}^\infty \rho^j (\frac{1-\theta}{B(\theta)})^j E_{\theta,j(1-\mu)+1}^{-\gamma j}(\lambda, t-t_0)\\
   &+& g(t)*\sum_{j=0}^\infty \rho^j (\frac{1-\theta}{B(\theta)})^{j+1} E_{\theta,(j+1)(1-\mu)}^{-\gamma (j+1)}(\lambda, t-t_0).
\end{eqnarray}
\end{lemma}

\section{Existence and uniqueness theorems}
Consider the system
\begin{equation}\label{systemm}
  ~_{t_0}^{ABC}D^{\theta,\mu,\gamma}x(t)=f(t,x(t)),~~x(t_0)=x_0,~t \in (t_0,T],
\end{equation}
where $\alpha \in (0,1)$, $f:[t_0,b)\times G$, $G\subseteq \mathbb{R} (~\texttt{or}~\mathbb{C})$ is open and
$$h(t)=f(t,x(t)) \in C[t_0,b]=\{y:[t_0,b]\rightarrow \mathbb{R}: y(t) \texttt{is continuous}\}.$$

The space $C[t_0,b]$ is a Banach space when it is endowed by the supremum norm
\begin{equation}\label{norm1}
  \|y\|_\infty=\sup_{t \in [t_0,T]}| y(t)|,
\end{equation}

\begin{definition}\label{solution}
A function $x(t)$ is said to be a solution of the initial value problem (\ref{systemm}) if
\begin{enumerate}
  \item $(t,x(t)) \in D,~~D=[t_0,T]\times B,~~B=\{x \in \mathbb{R}:|x|\leq L\}\subset G,~~L>0$
  \item $x(t)$ satisfies (\ref{systemm}).
\end{enumerate}
\end{definition}
\begin{theorem}\label{existence uniqueness}
The generalized ABC-fractional initial value problem  (\ref{systemm}) has a unique solution in the space
$$
C^{\theta,\mu,\gamma}[t_0,T]=\{y(t)\in C[t_0,b]:~_{t_0}^{ABC}D^{\theta,\mu,\gamma} y(t) \in C[t_0,T]\},
$$
 with $0 \leq \theta  <1, \mu,\gamma>0$, provided that
\begin{equation}\label{existence condition1}
C_1(T,A)=\frac{A}{B(\theta)} \sum_{i=0}^\infty \frac{\binom{\gamma}{i} \theta^i (T-t_0)^{i \theta -\mu+1}}{(1-\theta)^{i-1}\Gamma(i \theta+2-\mu)}<1,~\texttt{if}~ \mu\neq1,
\end{equation}
and
\begin{equation}\label{existence condition11}
C_2(T,A)=\frac{A}{B(\theta)}\left[(1-\theta)+ \sum_{i=1}^\infty \frac{\binom{\gamma}{i} \theta^i (T-t_0)^{i \theta }}{(1-\theta)^{i-1}\Gamma(i \theta+1)}\right]<1,~\texttt{if}~ \mu=1,
\end{equation}

and that $f$ satisfies the Lipschitz condition
\begin{equation}\label{Lipsch}
  |f(t,y_1)-f(t,y_2)|\leq A |y_1-y_2|,~~~A>0.
\end{equation}
Moreover, the case $\mu=1$ requires that $f(t_0,x(t_0))=0$.
\end{theorem}
\begin{proof}
 Define the operator $\Upsilon:C[t_0,T]\rightarrow C[t_0,T]$ by
\begin{equation}\label{operator}
 ( \Upsilon x)(t)=x_0 +~_{t_0}^{AB}I^{\theta,\mu,\gamma}f(t,x(t)),
\end{equation}
where the space $C[t_0,T]$  having  the  norm $\|.\|_\infty$.
For any $x_1,x_2 \in B$, by the help of the  Lipschitzian condition (\ref{Lipsch}) and by straight forward calculations, for any $t \in [t_0,T]$,  we have
\begin{equation}\label{calc}
  |(\Psi y_1(t)-\Psi y_2(t))|\leq  C_1(T,A)\|x_1-x_2\|_\infty,~~\texttt{if}~\mu\neq 1,
\end{equation}
and
\begin{equation}\label{calc1}
  |(\Psi y_1(t)-\Psi y_2(t))|\leq  C_2(T,A)\|x_1-x_2\|_\infty,~~\texttt{if}~\mu= 1.
\end{equation}

Then, taking the supremum over all $t\in[t_0,T]$ and using the assumptions (\ref{existence condition1}),(\ref{existence condition11}) we conclude that $\Upsilon$ a contraction on the Banach space $C[t_0,T]$.
 Therefore, there exists  a unique fixed point $x \in C[t_0,T]$ due to  Banach fixed point theorem. In addition,
\begin{equation}\label{more1}
  \lim_{m\rightarrow\infty}\|\Upsilon^mx_0-x\|_\infty=0.
\end{equation}
Because of  the definition of $\Upsilon$, $x$ possesses the form
\begin{equation}\label{srep}
  x(t)=x_0+~_{t_0}^{AB}I^{\theta,\mu.\gamma}f(t,x(t)).
\end{equation}
By  Lemma \ref{lem1}, Lemma \ref{lem2},  the identity  $$~^{ABR}_{t_0}D^{\theta,\mu,\gamma}~^{AB}_{t_0}I^{\theta,\mu,\gamma}\eta(t)=\eta(t),$$ and taking into account that $f(a,y(a))=0$ in the case $\mu=1$ , it can be shown that $y(t)$ satisfies the system (\ref{systemm}) if and only if it satisfies (\ref{operator}). Finally,  we have the estimate
  \begin{eqnarray}
           \nonumber
            \|~_{t_0}^{ABC}D^{\theta,\mu,\gamma}T^m x_0-~_{t_0}^{ABC}D^{\theta,\mu,\gamma}x\|_\infty &\leq& A\|\Upsilon^mx_0-x\|_\infty  \\  \nonumber
          \end{eqnarray}
          From (\ref{more1}), we conclude that $\lim_{m\rightarrow \infty} \|~_{t_0}^{ABC}D^{\theta,\mu,\gamma}T^m x_0-~_{t_0}^{ABC}D^{\theta,\mu,\gamma}x\|_\infty=0$. That is
          $~_{t_0}^{ABC}D^{\theta,\mu,\gamma}x\in C[t_0,T]$ and hence $x \in C^{\theta,\mu,\gamma}[t_0,T]$.

          The condition  $f(t_0,x(t_0))=0$ in case $\mu=1$ is needed in order to guarantee that solution given by (\ref{srep}) will satisfy $x(t_0)=x_0$. However, one may note that when $\mu\neq 1$ then $x(t_0)=x_0$ without any restrictions.
\end{proof}

\begin{remark}
The successive approximation generated in the proof of Theorem \ref{existence uniqueness} above were used in \cite{ThChaosAIP} to produce explicit solutions for the linear system (\ref{clinear})  above, by benefiting from the semigroup properties proved for the generalized AB-integral operators there. The Laplace transforms also were used there in \cite{ThChaosAIP} and completed in \cite{T1} for both the continuous and discrete cases to obtain an explicit solution as stated above in Lemma \ref{tool1}.
\end{remark}
\begin{theorem}\label{Cor1}
(\ref{mm1}) owns a unique solution in the space $C^{\theta,\mu,\gamma}[t_0,T]$ , provided that
\begin{equation}\label{existence condition}
C_1(T,A)<1,\texttt{if}~\mu\neq 1 \texttt{ and }C_2(T,A)<1,\texttt{if}~\mu= 1,
\end{equation}
where $A= r(1+2L)$. The case $\mu=1$ requires that either $x(t_0)=0$ or $x(t_0)=1$.
\end{theorem}
\begin{proof}
The proof follows from considering  Theorem \ref{existence uniqueness} with $f(t,x(t))=r x(t)(1-x(t))$ and by noting that $f$ satisfies the Lipschitz constant $A= r(1+2L)$ due to the fact that
 $$|f(t,x_1)-f(t,x_2)|=r|(x_1-x_2)(1+x_1+x_2)|\leq r(1+2L)|y_1-y_2|.$$
\end{proof}
\begin{theorem}\label{Cor2}
 (\ref{mm2}) acquires a unique solution in the space$C^{\theta,\mu,\gamma}[t_0,T]$ , provided that
\begin{equation}\label{existence condition}
C_1(T,A)<1,\texttt{if}~\mu\neq 1 \texttt{ and }C_2(T,A)<1,\texttt{if}~\mu= 1,
\end{equation}
where $A= r\left(-m+(1+\frac{m}{k})2L+\frac{L^2}{k} \right)$. The case $\mu=1$ requires that either $x(t_0)=0$ or $x(t_0)=m$ or $x(t_0)=k$.
\end{theorem}

\begin{proof}
The proof can be carried out by  applying  Theorem \ref{existence uniqueness} with  $f(t,x(t))=r x(t) (1-\frac{x(t)}{k})(x(t)-m)$ and then tracing the same steps as in the proof of Theorem 4 in \cite{TQF1}.
\end{proof}

\section{Stability analysis for the $ABC-$logistic models}
In this section  we dispute  the stability of the quadratic and cubic  logistic models by the help of  perturbation of  the equilibrium points.
Assume $\theta \in (0,1], \mu>0 ,~~\gamma >0$ and consider the $ABC-$fractional initial value problem

\begin{equation}\label{IVP}
  ~_{t_0}^{ABC}D^{\theta,\mu,\gamma}x(t)= f(x),~~x(t_0)=x_0,~~t>t_0.
\end{equation}
Let $y_0$ be an equilibrium point of the system (\ref{IVP}), that is  $f(y)=0$, and assume that $x(t)=y_0+\alpha(t)$. Since the ABC fractional derivative of the constant function is zero,  following the same as in Section 4 in \cite{TQF1}, we deduce that the system (\ref{IVP}) has the following corresponding perturbed system:

\begin{equation}\label{ps}
  ~_{t_0}^{ABC}D^{\theta,\mu,\gamma}\alpha(t)=f^\prime(y_0)\alpha(t),~~t>t_0,~~\alpha(t_0)=x_0-y_0.
\end{equation}
In what follow, we use the perturbed system (\ref{ps}) to study the stability of the logistic models by making use of Lemma \ref{tool1} with $\rho=f^\prime(y_0)$ and $g(t)=0$, where $f$ is the right hand side of the investigated logistic model.
\subsection{Analysis of the quadratic logistic model}
 We see that the quadratic logistic model  has the equilibrium points $y=0,1$. The corresponding  right hand side function of model (1) is $f(x)=rx(1-x)$ and hence
 $f^\prime(x)=r(1-2x)$ and $f^\prime(0)=r,~~f^\prime(1)=-r$.

 The perturbed system associated to the equilibrium point $z=0$ is the fractional linear system
 \begin{equation}\label{equi1}
    ~_{t_0}^{ABC}D^{\theta,\mu,\gamma}\alpha(t)=r \alpha(t),~~\alpha(t_0)=x_0.
 \end{equation}
Applying Lemma \ref{tool1} with $g(t)=0$  and $\rho=r$, the solution of system (\ref{equi1}) is given by
$$\alpha(t)=x_0+ x_0\sum_{j=1}^\infty r^j (\frac{1-\theta}{B(\theta)})^j E_{\theta,j(1-\mu)+1}^{-\gamma j}(\lambda, t-t_0),$$
and hence the equilibrium point $y=0$ is unstable.\\

In addition, the perturbed system associated to the equilibrium point $z=1$ is the fractional linear system
 \begin{equation}\label{equi2}
   ~_{t_0}^{ABC}D^{\theta,\mu,\gamma}\alpha(t)=-r\alpha(t),~~\alpha(t_0)=x_0-1.
 \end{equation}
The solution of system (\ref{equi2}) is given by
$$\alpha(t)=(x_0-1)+ (x_0-1)\sum_{j=1}^\infty (-r)^j (\frac{1-\theta}{B(\theta)})^j E_{\theta,j(1-\mu)+1}^{-\gamma j}(\lambda, t-t_0),$$
and hence the equilibrium point $y=1$ is asymptotically stable.
\subsection{Analysis of the cubic logistic model }
We see that the cubic logistic model  has the equilibrium points $y_1=0,y_2=m$ and $y_3=k$. The corresponding right hand side function of  (\ref{mm2}) is $f(x)=r x(t) (1-\frac{x(t)}{k})(x(t)-m)$ and hence
 $f^\prime(x)=r (1-\frac{x(t)}{k})(x(t)-m)+r x(t)(1-\frac{x(t)}{k})-\frac{r}{k}x(t)(x(t)-m)$ and $f^\prime(0)=-rm,~~f^\prime(m)=rm (1-\frac{m}{k})$, and $f^\prime(k)=-r (k-m)$.

 The perturbed system associated to the equilibrium point $y=0$ is the fractional linear system
 \begin{equation}\label{2equi1}
   ~_{t_0}^{ABC}D^{\theta,\mu,\gamma}\alpha(t)=-rm \alpha(t),~~\alpha(t_0)=x_0,~~t>t_0.
 \end{equation}
Applying Lemma \ref{tool1} with $g(t)=0$ and $\rho=rm$, the solution of system (\ref{2equi1}) is given by
$$\alpha(t)=x_0+ x_0\sum_{j=1}^\infty (-rm)^j (\frac{1-\theta}{B(\theta)})^j E_{\theta,j(1-\mu)+1}^{-\gamma j}(\lambda, t-t_0).$$
Since $r, m>0$ the equilibrium point $y_1=0$ is asymptotically stable.

Also the perturbed system associated to the equilibrium point $y_2=m$ is the fractional linear system
 \begin{equation}\label{2equi2}
   ~_{t_0}^{ABC}D^{\theta,\mu,\gamma}\alpha(t)=rm (1-\frac{m}{k})\alpha(t),~~\alpha(t_0)=x_0-m.
 \end{equation}
The solution of system (\ref{2equi2}) is given by
$$\alpha(t)=   (x_0-m)+ (x_0-m)\sum_{j=1}^\infty (rm (1-\frac{m}{k}))^j (\frac{1-\theta}{B(\theta)})^j E_{\theta,j(1-\mu)+1}^{-\gamma j}(\lambda, t-t_0).$$
Since $r,m,k, \rho>0$ and $m<k$, then the equilibrium point $y_2=m$ is unstable.

Finally, the perturbed system associated to the equilibrium point $y_3=k$ is the fractional linear system
 \begin{equation}\label{2equi3}
  ~_{t_0}^{ABC}D^{\theta,\mu,\gamma}\alpha(t)= -r(k-m) \alpha(t),~~\alpha(t_0)=x_0-k.
 \end{equation}
The solution of system (\ref{2equi3}) is given by
$$\alpha(t)=(x_0-k)+ (x_0-k)\sum_{j=1}^\infty (-r(k-m))^j (\frac{1-\theta}{B(\theta)})^j E_{\theta,j(1-\mu)+1}^{-\gamma j}(\lambda, t-t_0).$$
Since $r,m,k>0$ and $m<k$, then the equilibrium point $z_3=k$ is asymptotically stable.

\begin{remark}
Upon Theorem \ref{Cor1} and Theorem \ref{Cor2} the quadratic logistic model and the cubic logistic model lead to the trivial cases for the case $\mu=1$. Hence, the case $\mu\neq1$ turns to be of more interest. In fact, when $\mu=1$ we have

\begin{itemize}
  \item Under the assumptions $x(t_0)=0$, $x(t_0)=1$ the equilibrium  points $y_1=0$ and $y_2=1$ of the quadratic logistic model are stable, respectively. Their corresponding perturbed linear system in this case will have the trivial solution.
  \item Under the assumptions $x(t_0)=0$, $x(t_0)=m$  and $x(t_0)=k$ the equilibrium  points $y_1=0$ and $y=m$ and  $y_3=k$ of the cubic logistic model are stable, respectively. Their corresponding perturbed linear system in this case will have the trivial solution.
\end{itemize}
\end{remark}

\section{Numerical Discussion}
In this section, we give a description of the numerical scheme to solve the initial value problem
\begin{equation}\label{systemmm}
  ~_{t_0}^{ABC}D^{\theta,\mu,\gamma}x(t)=f(t,x(t)),~~x(t_0)=x_0,~t \in (t_0,T],
\end{equation}
where $\theta \in (0,1)$.

\subsection{ Numerical Scheme}
The scheme is a predictor-corrector method based on Lagrange interpolation. We derive two  schemes, one explicit (the predictor scheme) and the other implicit (the corrector scheme).

The solution of  (\ref{systemmm}) is written as
\begin{equation}
x(t)=x(t_0)+\sum_{k=0}^{\infty} A_k^{\theta,\mu,\gamma} (I_k(x))(t),
\end{equation}
where
\begin{eqnarray*}
 && (I_k(x))(t)=\int_{t_0}^t (t-s)^{\theta k-\mu}f(s,x(s))ds,\\
 && A^{\theta,\mu,\gamma}_k={\gamma\choose{k}} \frac{\theta^k}{B(\theta)(1-\theta)^{k-1}}\frac{1}{\Gamma(\theta k-\mu+1)}.
\end{eqnarray*}
The interval $(0,T]$ is discritised uniformly with grid points $t_m=t_0+m h,\ m\ge0$,  $h$ a step size. Then at $t_m$, we have
\be\label{xtm}
x(t_m)=x(t_0)+\sum_{k=0}^{\infty} A_k^{\theta,\mu,\gamma} (I_k(x))(t_m).
\ee
Then integral for $ (I_k(x))(t_m)$ is written as
$$
 (I_k(x))(t_m)=\sum_{j=0}^{m-1} I^m_{k,j}(x),\qquad  I^m_{k,j}(x)=\int_{t_j}^{t_{j+1}} (t_m-s)^{\theta k-\mu}f(s,x(s))ds.
$$
Here, depending on how we approximate  $I^m_{k,j}(x)$, we obtain either an implicit scheme or an explicit one. For an explicit scheme, we approximate $f(s,x(s))$ by a Lagrange polynomial using the nodes $t_j$ and $t_{j-1}$:
$$
f(s,x(s))\approx \frac{1}{h}[(s-t_{j-1}) f_j -(s-t_j)f_{j-1}],\qquad f_j=f(t_j,x(t_j)).
$$
This gives the following approximation for $I^m_{k,j}(x),\ j\ge 1$:
$$
I^m_{k,j}\approx  h^{\theta k-\mu+1}(f_j w^{m,1}_{k,j}-f_{j-1} w^{m,0}_{k,j}),
$$
where
$$
w^{m,i}_{k,j}=(m-j+i)\frac{(m-j)^{\theta k -\mu+1}-(m-j-1)^{\theta k -\mu+1}}
{\theta k-\mu+1}+\frac{(m-j-1)^{\theta k -\mu+2}-(m-j)^{\theta k -\mu+2}}{\theta k-\mu+2}.
$$
For $j=0$, we approximate $I^m_{k,0}$ by approximating $f(s,x(s))$ in the integrand by $f_0=f(t_0,x(t_0))$. This gives the approximation
$$
I^m_{k,0}\approx h^{\theta k-\mu+1}\xi^m_{k} f_0,\qquad\xi^m_{k}= \frac{m^{\theta k-\mu+1}-(m-1)^{\theta k-\mu+1}}{\theta k-\mu+1}.
$$
Therefore, $(I_k(x))(t_m), \ m\ge1$, is then approximated by
\bea
(I_k(x))(t_m)&=&  I^m_{k,0}(x)+ \sum_{j=1}^{m-1} I^m_{k,j}(x)\nonumber\\
&\approx& h^{\theta k-\mu+1}\left[\xi^m_{k} f_0+\sum_{j=1}^{m-1}
(f_j w^{m,1}_{k,j}-f_{j-1} w^{m,0}_{k,j})\right]\nonumber\\
&=& h^{\theta k-\mu+1} \sum_{j=0}^{m-1}
r^m_{k,j} f_j,\label{Iktm}
\eea
where
$$
r^m_{k,j}=\left\{\ba{ll} \xi^m_{k}-w^{m,0}_{k,1},\ &\mbox{if } j=0,\\
w^{m,1}_{k,j}-w^{m,0}_{k,j+1},\ & \mbox{if } 1\le j\le m-2,\\
w^{m,1}_{k,m-1},\ &  \mbox{if }  j=m-1.\ea\right.
$$
Substituting (\ref{Iktm}) into (\ref{xtm}), letting $x_m\approx x(t_m)$, we obtain the following  explicit scheme:
$$
x_m=x_0+\sum_{k=0}^{\infty} A_k^{\theta,\mu,\gamma} \left[h^{\theta k-\mu+1} \sum_{j=0}^{m-1}
r^m_{k,j} f_j\right]
$$
which can be written as
\be\label{explicitscheme}
x_m=x_0+\sum_{j=0}^{m-1} c^m_j f_j,\qquad c^m_j=\sum_{k=0}^{\infty} \left(h^{\theta k-\mu+1}A_k^{\theta,\mu,\gamma}r^m_{k,j}\right)
\ee
Note that if $\gamma$ is an integer, say $\gamma=K$, then $A_k^{\theta,\mu,\gamma}=0$ for $k>K$, and  $c^m_j$ becomes
$ c^m_j=\sum\limits_{k=0}^{K} \left(h^{\theta k-\mu+1}A_k^{\theta,\mu,\gamma}r^m_{k,j}\right)$.

For an implicit scheme, we approximate $f(s,x(s))$ in the integral defining $I^m_{k,j}(x)$ using the nodes $t_j$ and $t_{j+1}$:
$$
f(s,x(s))\approx \frac{1}{h}[(s-t_{j}) f_{j+1} -(s-t_{j+1})f_{j}],
$$
which gives the following approximation for $I^m_{k,j}(x),\ j\ge 1$:
$$
I^m_{k,j}\approx  h^{\theta k-\mu+1}(f_{j+1} w^{m,0}_{k,j}-f_{j} w^{m,-1}_{k,j}),
$$
Therefore, $(I_k(x))(t_m), \ m\ge1$, is then approximated by
\bea
(I_k(x))(t_m)&=&  \sum_{j=0}^{m-1} I^m_{k,j}(x)
\approx h^{\theta k-\mu+1}\sum_{j=0}^{m-1}
(f_{j+1} w^{m,0}_{k,j}-f_{j} w^{m,-1}_{k,j})\nonumber\\
&=& h^{\theta k-\mu+1} \sum_{j=0}^{m}
\tilde{r}^m_{k,j} f_j,\label{Iktmimplicit}
\eea
where
$$
\tilde{r}^m_{k,j}=\left\{\ba{ll} -w^{m,-1}_{k,0},\ &\mbox{if } j=0,\\
w^{m,0}_{k,j-1}-w^{m,-1}_{k,j},\ & \mbox{if } 1\le j\le m-1,\\
w^{m,0}_{k,m-1},\ &  \mbox{if }  j=m.\ea\right.
$$
Substituting (\ref{Iktmimplicit}) into (\ref{xtm}), leads to the following  implicit scheme:
\be\label{implicitscheme}
x_m=x_0+\sum_{j=0}^{m} \tilde{c}^m_j f_j,\qquad \tilde{c}^m_j=\sum_{k=0}^{\infty} \left(h^{\theta k-\mu+1}A_k^{\theta,\mu,\gamma}\tilde{r}^m_{k,j}\right).
\ee

\subsection{ Numerical Simulations}
For sake of simplicity, we will focus on the numerical discussion of Equation (1) given by:
\begin{equation}\label{NS-mm1}
   ( ~_{t_0}^{ABC}D^{\alpha,\mu,\gamma}x)(t)=r x(t) (1- x(t)/K),~~t>t_0,~x(t_0)=x_0,
\end{equation}
with $r=0.5$ and $K=2$. The targets of this example are to discuss the effects of the parameters $x_0$, $\theta$, $\mu$, and $\gamma$ on the solution trajectories.  \\
\\
It is clearly observed that  equation (\ref{NS-mm1}) has two equilibria given by $x_1=0$ and $x_2=2$. Figure \ref{fig-m1-n-1} shows the solution trajectories as the initial point, at $x_0$, changes in the set $\{0.5, 1, 1.5, 2.5, 3\}$ when $\gamma = 1$, $\alpha = 0.5$, and $\mu = 0.5$.  One can clearly see that  the solution trajectories converge to $x_2=2$ asymptotically for any $x_0$. Thus, we conclude that $x_2=2$ is  asymptotically stable equilibrium solution whereas $x_1=0$ is unstable equilibrium solution. The rate of convergence of  solution trajectories to the equilibrium solutions is strongly dependent on the initial points; i.e. it  is higher as the initial point closer to the value of the steady state, $x_2=2$. It is worth mentioning that the behaviour of solution trajectories in this case is similar to the integer case.
\begin{figure}[h]
\begin{center}
\includegraphics[height=7cm]{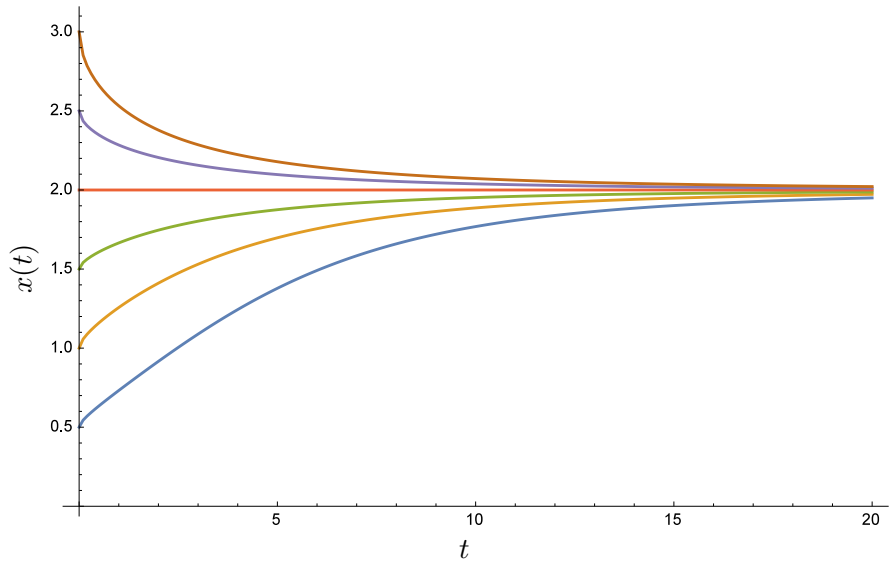}
\end{center}
\begin{picture}(5,5)(5,5)
                \put(210,20){$t$}
                \put(40,120){\rotatebox{90}{$x(t)$}}
    \end{picture}
\caption{Graphs of the the solution trajectories for Example 1 at $\gamma = 1$, $\alpha = 0.5$, and $\mu = 0.5$ for different values of initial condition.} \label{fig-m1-n-1}
\end{figure}
\\
The effect of changing $\gamma$ from $1.0$ to $3.0$ at fixed values of $\alpha=1/2$ and $\mu=1/2$ on the behaviour of solution trajectories is presented in Figure \ref{fig-m1-n-2}. We consider two values for $x_0$: $1.0$ and $3.0$. Obviously, the required time for solution trajectories  to reach the equilibrium solution $x_2=2$ decreases as $\gamma$ increases. An interesting phenomenon is the oscillatory behavior at early stages of the time appeared as $\gamma$ increases.  \\
\begin{figure}[h]
\begin{center}
\includegraphics[height=7cm]{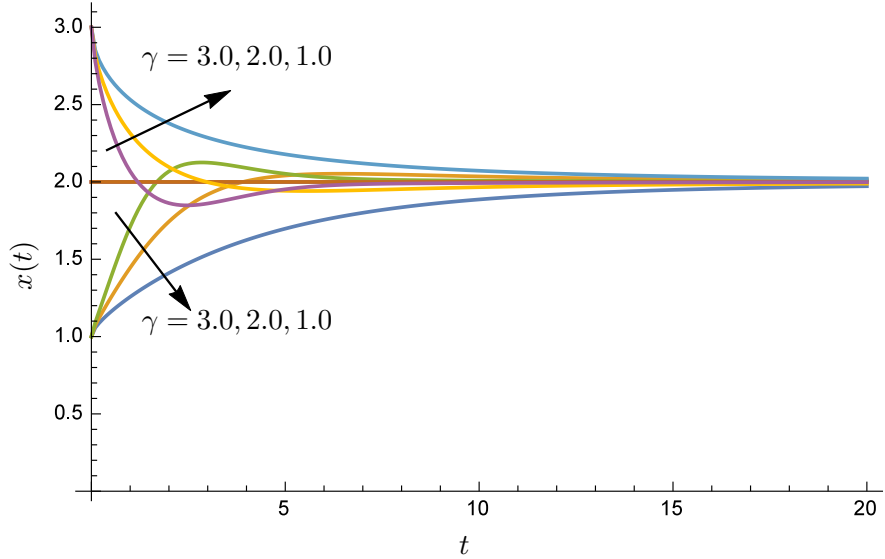}
\end{center}
\begin{picture}(5,5)(5,5)
                \put(210,20){$t$}
                \put(40,120){\rotatebox{90}{$x(t)$}}
                \put(90,105){$\gamma=3.0, 2.0, 1.0$}
                \put(90,205){$\gamma=3.0, 2.0, 1.0$}
    \end{picture}
\caption{Graphs of the  solution trajectories for Example 1 at $x_0=1.0$, $x_0=3.0$, $\alpha=1/2$ and $\mu=1/2$ for different values of $\gamma=1.0, 2.0,3.0$.} \label{fig-m1-n-2}
\end{figure}
\\
Figure \ref{fig-m1-n-3} shows the solution trajectories at $\alpha=1/2$, $\gamma=1.0$ and $x_0=1.0$ and $3.0$, while $\mu$ is changing in the set $\{0.4, 0.6, 0.8\}$. we observe that the required time for solution trajectories  to reach the equilibrium solution decreases as $\mu$ decreases.
\begin{figure}[h]
\begin{center}
\includegraphics[height=7cm]{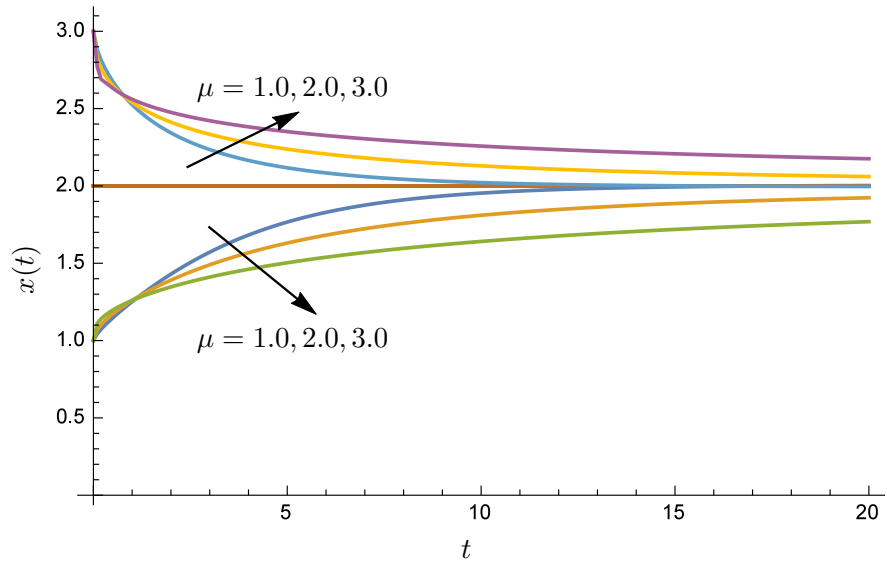}
\end{center}
\begin{picture}(5,5)(5,5)
                \put(210,20){$t$}
                \put(40,120){\rotatebox{90}{$x(t)$}}
                \put(110,100){$\mu=1.0, 2.0, 3.0$}
                \put(110,195){$\mu=1.0, 2.0, 3.0$}
    \end{picture}
\caption{Graphs of the the solution trajectories for Example 1 at $x_0=1.0$, $x_0=3.0$, $\alpha=1/2$ and $\gamma=1.0$ for different values of $\mu=0.4, 0.6, 0.8$.} \label{fig-m1-n-3}
\end{figure}
\\
The effect of changing $\alpha$ in the set $\{0.7, 0.9, 0.99 \}$ at fixed values of $\gamma=1$ and $\mu=1/2$ on the behaviour of solution trajectories is presented in Figure \ref{fig-m1-n-4}. We consider two values for $x_0$: $1.0$ and $3.0$. It is noted that the effect of increasing $\alpha$ will slow the required time for solution trajectories  to reach the equilibrium solution. Once again, the oscillatory behaviour as $\alpha$ increases is captured. \\
\begin{figure}[h]
\begin{center}
\includegraphics[height=7cm]{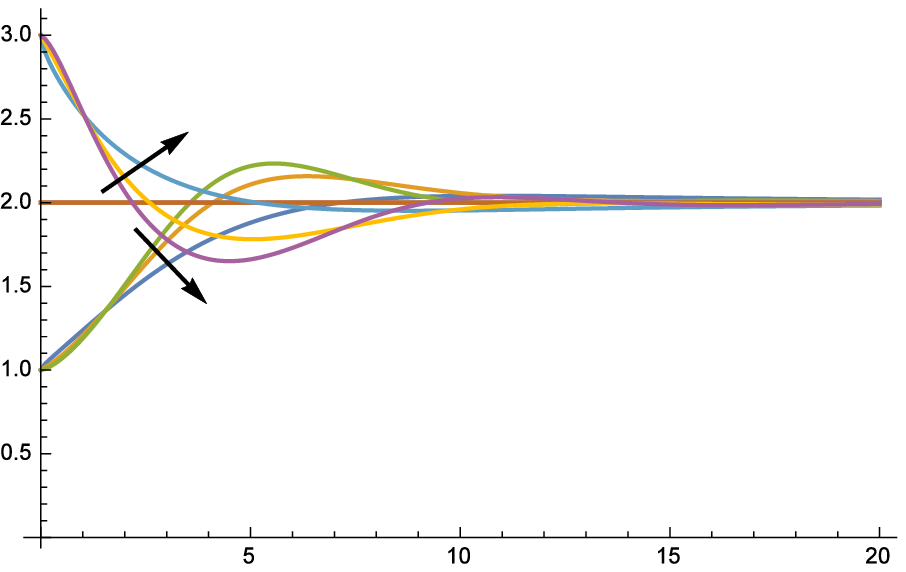}
\end{center}
\begin{picture}(5,5)(5,5)
                \put(210,20){$t$}
                \put(40,120){\rotatebox{90}{$x(t)$}}
                \put(90,105){$\alpha=0.99, 0.9, 0.7$}
                \put(90,205){$\alpha=0.99, 0.9, 0.7$}
    \end{picture}
\caption{Graphs of the the solution trajectories for Example 1 at $x_0=1.0$, $x_0=3.0$, $\mu=1/2$ and $\gamma=1.0$ for different values of $\alpha=0.7, 0.9, 0.999$.} \label{fig-m1-n-4}
\end{figure}
\\

\section{Conclusion}
In this article, we analysed the logistic equations in the setting of ABC-fractional  derivatives with generalized Mittag-Leffler kernels. Such kind of a  fractional derivative contains two parameters $\gamma$ and $\mu$ along with the order $\alpha$. We discussed the Existence and uniqueness condition in addition to their stability. Further, numerical examples were considered to demonstrate these results. It is clearly seen that for fixed values of the parameters  $\alpha,\mu$ and $\gamma$, the convergence of the solution to the equilibrium point is dependent on the initial value. For fixed initial value and fixed $\alpha$ and $\mu$, the convergence of the solution of to the equilibrium point is faster when greater values of $\gamma$ are considered while for fixed values of $\alpha,\gamma$,  the convergence of the solution of to the equilibrium point is faster for smaller values of $\mu$. Morever, it can be obviously observed  that for greater values of $\alpha$, the required time for solution trajectories  to reach the equilibrium solution decreases.

\end{document}